\documentclass[11pt]{amsart}
\usepackage{amsmath,amsthm,amsfonts,amssymb,bm,graphicx}
\usepackage[alphabetic,initials]{amsrefs}
\usepackage{color}
\usepackage
{hyperref}
\hypersetup{colorlinks=true,citecolor=blue,linkcolor=blue,urlcolor=blue,
pdfstartview=FitH }

\usepackage[margin=1.1in]{geometry}

\theoremstyle{plain}
\newtheorem{theorem}{Theorem}[section]

\newtheorem{conjecture}{Conjecture}[section]
\newtheorem{lemma}{Lemma}[section]
\newtheorem{proposition}{Proposition}[section]

\theoremstyle{definition}

\renewcommand{\geq}{\geqslant}
\renewcommand{\leq}{\leqslant}

\newcommand{\dd}{\mathrm{d}}

\numberwithin{equation}{section}

\makeatletter
\def\Ddots{\mathinner{\mkern1mu\raise\p@
\vbox{\kern7\p@\hbox{.}}\mkern2mu
\raise4\p@\hbox{.}\mkern2mu\raise7\p@\hbox{.}\mkern1mu}}
\makeatother
\makeatletter
\DeclareRobustCommand\widecheck[1]{{\mathpalette\@widecheck{#1}}}
\def\@widecheck#1#2{%
    \setbox\z@\hbox{\m@th$#1#2$}%
    \setbox\tw@\hbox{\m@th$#1%
       \widehat{%
          \vrule\@width\z@\@height\ht\z@
          \vrule\@height\z@\@width\wd\z@}$}%
    \dp\tw@-\ht\z@
    \@tempdima\ht\z@ \advance\@tempdima2\ht\tw@ \divide\@tempdima\thr@@
    \setbox\tw@\hbox{%
       \raise\@tempdima\hbox{\scalebox{1}[-1]{\lower\@tempdima\box
\tw@}}}%
    {\ooalign{\box\tw@ \cr \box\z@}}}
\makeatother

\begin{document}

\title[The distribution of powers of primes]
{The distribution of powers of primes related to the Frobenius problem}

\author{Enxun Huang}

\author{Tengyou Zhu}

\address{School of Mathematics, Shandong University
                       \\Jinan, Shandong 250100, China}

\email{huangenxun@mail.sdu.edu.cn}

\email{zhuty@mail.sdu.edu.cn}	

\date{}



\begin{abstract}
Let $1<c<d$ be two relatively prime integers, $g_{c,d}=cd-c-d$
and $\mathbb{P}$ is the set of primes.
For any given integer $k \geq 1$, we prove that
$$\#\left\{p^k\le g_{c,d}:p\in \mathbb{P},
~p^k=cx+dy,~x,y\in \mathbb{Z}_{\geqslant0}
\right\}\sim \frac{k}{k+1}\frac{g^{1/k}}{\log g}
\quad (\text{as}~c\rightarrow\infty),$$
which gives an extension of a recent result of Ding, Zhai and Zhao.
\end{abstract}

\maketitle

\section{Introduction}\label{intr.}
Given two coprime integers $1< c<d$, Sylvester~\cite{SJJ} discovered in 1882 that $cd-c-d$ represents the largest number, denoted as $g_{c,d}$, which cannot be expressed as $cx+dy~(x,y\in \mathbb{Z}_{\geqslant0})$. Subsequently, the author demonstrated that for any $0\leq m\leq g_{c,d}$, exactly one of $m$ and $g_{c,d}-m$ can be expressed in the form $cx+dy~(x,y\in \mathbb{Z}_{\geqslant0})$. Consequently, it can be concluded that precisely half of the integers within the range $[0,g_{c,d}]$ can be represented in this desired form.

 Indeed, Sylvester's findings represent a significant breakthrough in the field of the diophantine Frobenius problem, marking the first instance of nontrivial solutions \cite{RA}. This problem seeks to find the largest integer $g_{c_1,...,c_n}$ not that can be expressed as $$c_1x_1+\cdots+c_nx_n \quad (x_1,...,x_n\in \mathbb{Z}_{\geqslant0}),$$given that $c_1,...,c_n$ are positive integers with $\gcd(c_1,...,c_n)=1$. There exists an extensive body of literature dedicated to studying the diophantine Frobenius problem. For comprehensive insights and findings on this topic, please refer to Ram\'{\i}rez Alfons\'{\i}n's exceptional monograph \cite{RA}.
Motivated by Sylvester's theorems, Ram\'{\i}rez Alfons\'{\i}n and Ska{\l}ba \cite{RS} investigated the Diophantine Frobenius problem restricted to prime numbers. Specifically, let $\pi_{c,d,k}$ denote the number of $k$-th powers of primes not exceeding $g_{c,d}$ with the form $cx+dy~(x,y\in \mathbb{Z}_{\geqslant0})$. Through a highly enlightening argument put forth by Ram\'{\i}rez Alfons\'{\i}n and Ska{\l}ba, it has been demonstrated that for any $\varepsilon>0$, there exists a constant $k_{\varepsilon}>0$ such that
$$\pi_{c,d,1}\geqslant k_{\varepsilon}\frac{g_{c,d}}{(\log g_{c,d})^{2+\varepsilon}}.$$
After observing the antisymmetry property of integers in the form of $cx+dy~(x,y\in \mathbb{Z}_{\geqslant0})$ discovered by Sylvester, they naturally formulated the following conjecture.

\begin{conjecture}[Ram\'{\i}rez Alfons\'{\i}n and Ska{\l}ba]\label{conjecture1}
Let $1< c<d$ be two relatively prime integers, then $$\pi_{c,d,1}\sim\frac{\pi(g_{c,d})}{2}\quad (\text{as}~c\rightarrow\infty),$$
where $\pi(t)$ is the number of primes up to $t$.
\end{conjecture}
The recent proof of Conjecture \ref{conjecture1} by Ding, Zhai, and Zhao in~\cite{DZZ} has established its validity. Prior to this, Ding~\cite{D} demonstrated that Conjecture \ref{conjecture1} holds true for almost all values of $c$ and $d$. Remarkably, the complete verification of Conjecture \ref{conjecture1} relies on the application of the classical Hardy-Littlewood method.

The main result is presented below. Throughout the remainder of this paper, let $k\geq 1$ be an integer.
For convenience, we write $g$ as $g_{c,d}$ and $\pi_{c,d}$ as $\pi_{c,d,k}.$
Now, we shall establish the following theorem to document our findings.

\begin{theorem}\label{thm1}
Suppose that $d>c$ are two relatively prime integers with $c$ sufficiently large, then we have
$$\pi_{c,d}\sim \frac{k}{k+1}\frac{g^{1/k}}{\log g}, \quad \text{as~}c\rightarrow\infty.$$
\end{theorem}
Theorem~\ref{thm1} presented herein extends the findings in~\cite{DZZ} to encompass arbitrary positive integers $k$. The proof methodology employed in this study aligns with a recent publication by Ding, Zhai, and Zhao~\cite{DZZ}.

The key input for the proof of Theorem~\ref{thm1} is the following result.
\begin{theorem}\label{thm2}
Let $c,d$ be as in Theorem~\ref{thm1}. Denote
\begin{equation*}
 \mathcal{N} := \#\{n^k \leq g_{c,d}: n^k=cx+dy,~x,y\in \mathbb{Z}_{\geqslant0}\},
\end{equation*}
then we have
$$\mathcal{N}  \sim \frac{g^{1/k}}{k+1}, \quad \text{as~}c\rightarrow\infty.$$
\end{theorem}

The initial investigation will follow the customary  and focus on the weighted form associated with Theorem~\ref{thm1}, specifically
\begin{align*}
\psi_{c,d} = \sum_{\substack{n^k\leq g\\ n^k=cx+dy\\ x,y\in \mathbb{Z}_{\geqslant0}}}\Lambda (n),
\end{align*}
where $\Lambda(n)$ denotes the von Mangoldt function defined as
\begin{equation*}
\Lambda(n)=
\begin{cases}
\log p, & \text{if } n=p^{\alpha}~(\alpha>0);\\
0, & \text{otherwise}.
\end{cases}
\end{equation*}
The proof of Theorem\ref{thm1} will be established using the following weighted formula through a standard transition method.

\begin{theorem}\label{thm3}
Let $c,d$ be as in Theorem~\ref{thm1},  then we have
$$\psi_{c,d}\sim \frac{g^{1/k}}{k+1}, \quad \text{as~}c\rightarrow\infty.$$
\end{theorem}
 The proofs heavily rely on the exponential sum estimates in~\cite{Kum}, the prime number theorem, and the Siegel--Welfisz theorem.

\section{First transformations}
We first fix some basic notations to be used frequently. From now on, we write $g$ instead of $g_{c,d}$ for brevity and $c$ is supposed to be sufficiently large. Let $Q$ denote a positive integer depending only on $g$ which shall be decided later.
Define the major arcs to be

\begin{align}\label{eq2-1}
\mathfrak{M}(Q)=\bigcup_{1\leq q\leq Q}\bigcup_{\substack{1\leq a\leq q\\ (a,q)=1}}\left\{\alpha:\left|\alpha-\frac{a}{q}\right|\leq \frac{Q}{qg}\right\}.
\end{align}
We make a further provision that $Q< (g/2)^{1/3}$ so that the above subsets are pairwise disjoint. In fact, suppose that
$$\left\{\alpha:\left|\alpha-\frac{a}{q}\right|\leq \frac{Q}{qg}\right\}\bigcap\left\{\alpha:\left|\alpha-\frac{a'}{q'}\right|\leq \frac{Q}{qg}\right\}\neq \emptyset$$
for some $\frac{a}{q}\neq \frac{a'}{q'}$, then
$$\frac{2Q}{g}\geq \frac{2Q}{qg}\geq \left|\frac{a}{q}-\frac{a'}{q'}\right|\geq \frac{1}{Q^2},$$
which is certainly a contradiction with the provision that $Q< (g/2)^{1/3}$. In addition, we note that
$$\mathfrak{M}(Q)\subseteq \left[\frac{1}{Q}-\frac{Q}{qg}, 1+\frac{Q}{qg}\right] \subseteq\left[\frac{Q}{g}, 1+\frac{Q}{g}\right].$$
We can now define the minor arcs to be
\begin{align}\label{eq2-2}
\mathfrak{m}(Q)=\left[\frac{Q+1}{g}, 1+\frac{Q+1}{g}\right]\setminus \mathfrak{M}(Q).
\end{align}
For any $\alpha \in \mathbb{R}$, let
$$
f(\alpha)=\sum_{0\leq n^k\leq g}\Lambda(n)e(\alpha n^k)
\quad
\text{and}
\quad h(\alpha)=\sum_{\substack{0\leq x\leq d
\\ 0\leq y \leq c}}e(\alpha (cx+dy)).
$$
By the orthogonality relation, it is clear that
\begin{align}\label{eq2-3}
\psi_{c,d} = \int_{0}^{1}f(\alpha)h(-\alpha)\dd\alpha=
\int_{\mathfrak{M}(Q)}f(\alpha)h(-\alpha)\dd\alpha+
\int_{\mathfrak{m}(Q)}f(\alpha)h(-\alpha)\dd\alpha.
\end{align}

\section{Estimates of the minor arcs}
The aim of this section is to prove the following proposition.
\begin{proposition}\label{proposition 3.1}
For estimates of the minor arcs, we have
\begin{align*}
\int_{\mathfrak{m}(Q)}f(\alpha)h(-\alpha)d\alpha\ll
g^{(1-\rho(k))/k+\varepsilon}\log^{4}g+g^{1/k}(\log g)^{\delta+4}
Q^{-\frac{1}{2}+\varepsilon},
\end{align*}
where $\varepsilon > 0$ is arbitrarily small.
\end{proposition}
We need two lemmas listed below.
\begin{lemma}\label{lemma3.1}
Let $k\ge 2$, and $\rho(k)$ defined as follows,
\begin{equation*}
\rho(k)=
\begin{cases}
\frac{1}{8}, & \text{if}  \ \; k=2;\\
\frac{1}{14}, & \text{if}  \ \; k=3;\\
\frac{2}{3}\times 2^{-k}, &\text{if} \  \; k\ge 4.
\end{cases}
\end{equation*}
Then for any fixed $\varepsilon >0$,
we have
\begin{equation*}
\sup_{\alpha\in \mathfrak{m}(Q)}|f(\alpha)|
\ll
\log^{2}g\big(g^{(1-\rho(k))/k+\varepsilon}+g^{1/k}(\log g)^\delta
Q^{-\frac{1}{2}+\varepsilon}\big ),
\end{equation*}
where $\delta >0$ is an absolute constant
and the implied constant depends at most
on $k$ and $\varepsilon$.
\end{lemma}
\begin{proof}
Plugging in the definition of $\Lambda(n)$, we have
\begin{equation*}
 f(\alpha)=\sum_{m=1}^{\infty}\sum_{p^m \leq g^{1/k}}\Lambda(p^m)e(\alpha p^{mk})
 =\sum_{m\leq \frac{\log g}{k\log 2}}\sum_{p \leq g^{1/{mk}}}\log p\,e\big(\alpha p^{mk}\big).
\end{equation*}
For $m\geq 2$, it is clear that
\begin{equation*}
 \sum_{2\leq m\leq \frac{\log g}{k\log 2}}\sum_{p \leq g^{1/{mk}}}
 \log p\,e\big(\alpha p^{mk}\big)\ll g^{1/{2k}}\log^2g.
\end{equation*}
Thus
\begin{equation}\label{f1}
 f(\alpha)=\sum_{p \leq g^{1/k}}\log g\,e\big(\alpha p^k\big)+O(g^{1/{2k}}\log^2g).
\end{equation}
On applying the partial summation formula, we have
\begin{equation}\label{f2}
\sum_{p \leq g^{1/k}}\log p\,e\big(\alpha p^k\big)=
\frac{\log g}{k}\sum_{p \leq g^{1/k}}e\big(\alpha p^k\big)-\int_{2}^{g^{1/k}}
\sum_{p \leq \gamma}e\big(\alpha p^k\big)\frac{\dd\gamma}{\gamma}.
\end{equation}
By the Dirichlet approximation theorem,
there exist $b\in \mathbb{Z}$ and $r\in \mathbb{N}$ such that
$$(b,r)=1, \quad 1\leq r\leq R \quad \text{and}
\quad \left|\alpha-\frac{b}{r}\right|\leq \frac{1}{rR},$$
where
\begin{equation*}
R=\begin{cases}
g^{3/2k}, & \text{if} \ \; k=2;\\
g^{(k-2\rho(k))/(2k-1)}, &\text{if}\ \; k\ge 3.
\end{cases}
\end{equation*}
Then by Theorem 3 in~\cite{Kum}, for any fixed $\varepsilon>0$,  we conclude
\begin{equation}\label{f bound1}
\sum_{g^{1/k} \leq p \leq 2g^{1/k}}e\big(\alpha p^k\big)\ll
g^{(1-\rho(k))/k+\varepsilon}
+\frac{r^\varepsilon g^{1/k}(\log g)^\delta}{(r+g|r\alpha-b|)^{1/2}},
\end{equation}
where $\delta>0$ is an absolute constant.
If $r > Q$, then
\begin{equation}\label{f bound2}
\sum_{g^{1/k} \leq p \leq 2g^{1/k}}e\big(\alpha p^k\big)\ll
g^{(1-\rho(k))/k+\varepsilon}+g^{1/k}(\log g)^\delta
Q^{-\frac{1}{2}+\varepsilon}.
\end{equation}
Similarly, if $1\leq r \leq Q$ and $|\alpha-b/r|> \frac{Q}{rg}$,
then~\eqref{f bound2} hold as well. From now on, we assume that
\begin{equation}\label{minor1}
1\leq  r \leq Q \quad \text{and}\quad
 \left|\alpha-\frac{b}{r}\right| \leq \frac{Q}{rg}.
\end{equation}
By the Dirichlet approximation theorem, there exist $a\in \mathbb{Z}$ and $q\in \mathbb{N}$ such that
\begin{equation}\label{minor2}
(a,q)=1, \quad 1\leq q\leq g^{1/2} \quad \text{and} \quad \left|\alpha-\frac{a}{q}\right|\le \frac{1}{qg^{1/2}}.
\end{equation}
We deduce from \eqref{minor1} and \eqref{minor2} that
\begin{align*}
\left|\frac{b}{r}-\frac{a}{q}\right|\leq &\left|\alpha-\frac{b}{r}\right|+\left|\alpha-\frac{a}{q}\right|
\\ \leq &\frac{Q}{rg} +\frac{1}{qg^{1/2}}
=\frac{1}{rq}\Big(\frac{qQ}{g} +\frac{r}{g^{1/2}}\Big)
\\ \leq &\frac{Q}{rqg^{1/2}} < \frac{1}{rq}.
\end{align*}
It follows that
\begin{align}\label{type1linkbraq1}
a=b \quad \text{and} \quad r=q
\end{align}
We obtain from~\eqref{f bound1} and~\eqref{type1linkbraq1} that
\begin{equation}\label{f bound3}
\sum_{g^{1/k} \leq p \leq 2g^{1/k}}e\big(\alpha p^k\big)\ll
g^{(1-\rho(k))/k+\varepsilon}
+\frac{r^\varepsilon g^{1/k}(\log g)^\delta}{(q+g|q\alpha-a|)^{1/2}},
\end{equation}
Since $\alpha\in\mathfrak{m}(Q)$, we further obtain
\begin{equation*}
\sum_{g^{1/k} \leq p \leq 2g^{1/k}}e\big(\alpha p^k\big)\ll
g^{(1-\rho(k))/k+\varepsilon}+g^{1/k}(\log g)^\delta
Q^{-\frac{1}{2}+\varepsilon}.
\end{equation*}
Then we deduce that
\begin{equation*}
\begin{split}
\sum_{p \leq g^{1/k}}e\big(\alpha p^k\big)&=\sum_{1\leq p\leq g^{1/k}/Q}
e\big(\alpha p^k\big)+ \sum_{g^{1/k}/Q<p\leq g^{1/k}}e\big(\alpha p^k\big)\\
&\ll \frac{g^{1/k}}{Q}+\sum_{1\leq \kappa \leq \frac{\log Q}{\log 2}}
\bigg|\sum_{\frac{g^{1/k}}{2^{\kappa-1}}
<p\leq \frac{g^{1/k}}{2^\kappa}}e\big(\alpha p^k\big)\bigg|\\
&\ll g^{(1-\rho(k))/k+\varepsilon}
+g^{1/k}(\log g)^\delta Q^{-\frac{1}{2}+\varepsilon}.
\end{split}
\end{equation*}
Employing this estimate, we deduce by~\eqref{f1} and~\eqref{f2} that
\begin{equation*}
\sup_{\alpha\in \mathfrak{m}(Q)}|f(\alpha)|
\ll g^{(1-\rho(k))/k+\varepsilon}\log^{2}g+g^{\frac{1}{k}}(\log g)^{\delta+2}
Q^{-\frac{1}{2}+\varepsilon}.
\end{equation*}

This completes the proof of Lemma~\ref{lemma3.1}.
\end{proof}
\begin{lemma}\label{lemma3.2}
We have
$$\int_{0}^{1}|h(-\alpha)|\dd\alpha\ll \log^2 g.$$
\end{lemma}
\begin{proof}
This is Lemma 3.2 in~\cite{DZZ}.
\end{proof}
\begin{proof}[Proof of Proposition \ref{proposition 3.1}]
The treatment of the minor arcs benefits from the following trivial estimates
$$\bigg|\int_{\mathfrak{m}(Q)}f(\alpha)h(-\alpha)\dd \alpha
\bigg|\leq \sup_{\alpha\in \mathfrak{m}(Q)}|f(\alpha)|\int_{0}^{1}|h(-\alpha)|\dd \alpha$$
together with Lemma~\ref{lemma3.1} and Lemma~\ref{lemma3.2}.
\end{proof}

\section{Calculations of the major arcs}

We now calculate the integral on the major
arcs.
\begin{proposition}\label{proposition 4.1}
For $Q< c^{1/3}$, we have
\begin{align*}
\int_{\mathfrak{M}(Q)}f(\alpha)h(-\alpha)\dd\alpha=\mathcal{N}
+O\Big(g^{1/k}Q^2\exp\big(-\kappa_1\sqrt{\log g}\big)+
g^{1/k}(\log g)^3Q^{-2^{1-k}+\varepsilon}+dQ^3\Big),
\end{align*}
where $\kappa_1>0$ is an absolute constant.
\end{proposition}

For any $\alpha \in \mathbb{R}$, we define
$$F(\alpha):=\sum_{0\leq n^k\leq g}e(\alpha n^k).$$
\begin{lemma}\label{major arcs1}
For any real number $\beta$, we have
\begin{equation*}
\int_{-\frac{Q}{g}}^{\frac{Q}{g}}f(\beta)h\left(-\beta\right)\dd\beta
=\int_{-\frac{Q}{g}}^{\frac{Q}{g}}F(\beta)h\left(-\beta\right)\dd\beta
+O\left(g^{1/k}Q^2\exp\left(-\kappa_1\sqrt{\log g}\right)\right),
\end{equation*}
where $\kappa_1>0$ is an absolute constant.
\end{lemma}
\begin{proof}
Let  $a_n=\Lambda(n)-1$. Then for any real number $\beta$, we have
\begin{equation*}
 f(\beta)-F(\beta)
=\sum_{0\leq n^k\leq g}a_ne\big(\beta n^k\big).
\end{equation*}
By partial summation, we have
\begin{equation*}
  \begin{split}
\sum_{0\leq n^k\leq g}a_ne\big(\beta n^k\big)
&=\sum_{0\leq n\leq g^{1/k}}a_ne\big(\beta n^k\big)\\
&=e\big(\beta g^k\big)\sum_{0\leq n\leq g^{1/k}}a_n-2\pi ik\beta\int_{0}^{g^{1/k}}\left(\sum_{0\leq n\leq t}a_n\right)
e\big(\beta t^k\big)t^{k-1}\dd t.
  \end{split}
\end{equation*}
Using the prime number theorem, there exists
some absolute constant $\kappa_1 > 0$ so that
$$
\sum_{0\leq n\leq t}a_n=\psi(t)-t
\ll t\exp\left(-\kappa_1\sqrt{\log t}\right).
$$
We deduce that
\begin{equation*}
f(\beta)=F(\beta)+O\left(g^{1/k}(1+k|\beta|g)
\exp\left(-\kappa_1\sqrt{\log g}\right)\right).
\end{equation*}
Then we arrive at
\begin{equation*}
\int_{-\frac{Q}{g}}^{\frac{Q}{g}}f(\beta)h\left(-\beta\right)\dd\beta
=\int_{-\frac{Q}{g}}^{\frac{Q}{g}}F(\beta)h\left(-\beta\right)\dd\beta
+\mathcal{R}(\beta),
\end{equation*}
where the error term $\mathcal{R}(\beta)$ can be bounded easily by the trivial estimates as
\begin{equation*}
\begin{split}
\mathcal{R}(\beta)&\ll \int_{-\frac{Q}{g}}^{\frac{Q}{g}}
g^{1/k}(1+k|\beta|g)\exp\left(-\kappa_1\sqrt{\log g}\right)|h(-\beta)|\dd\beta\\
&\ll\int_{-\frac{Q}{g}}^{\frac{Q}{g}}g^{1/k}(1+k|\beta|g)
\exp\left(-\kappa_1\sqrt{\log g}\right)g\dd\beta\\
&\ll g^{1+1/k}Q\exp\left(-\kappa_1\sqrt{\log g}\right)\int_{-\frac{Q}{g}}^{\frac{Q}{g}}1\dd\beta\\
&\ll g^{1/k}Q^2\exp\left(-\kappa_1\sqrt{\log g}\right).
\end{split}
\end{equation*}
This completes the proof of Lemma~\ref{major arcs1}.
\end{proof}
\begin{lemma}\label{major arcs2}
Let $2\leq q \leq Q$, $(a,q)=1$ and $|\beta|\leq Q/(qg)$. For $Q < c^{1/3}$, we have
\begin{align*}
h\left(-\frac{a}{q}-\beta\right)\ll qd.
\end{align*}
\end{lemma}
\begin{proof}
Recall that
\begin{align*}
	h(-\alpha)=\sum_{x\leq d}e(-\alpha cx)\sum_{y\leq c}e(-\alpha dy)\ll \min\{d, ||c\alpha||^{-1}\}\min\{c, ||d\alpha||^{-1}\},
\end{align*}
hence we have
\begin{align*}
h\left(-\frac{a}{q}-\beta\right)\ll \min\left\{d, \left\Vert c\left(-\frac{a}{q}-\beta\right)\right\Vert^{-1}\right\}\min\left\{c, \left\Vert d\left(-\frac{a}{q}-\beta\right)\right\Vert^{-1}\right\}.
\end{align*}
Since $(a,q)=1$, we would have
\begin{align*}
 \left\Vert c\left(-\frac{a}{q}-\beta\right)\right\Vert\geq \frac{1}{2q} \quad \text{if} \quad q\nmid c
\end{align*}
and
\begin{align*}
\left\Vert d\left(-\frac{a}{q}-\beta\right)\right\Vert\geq \frac{1}{2q} \quad \text{if} \quad q\nmid d
\end{align*}
for $q\ge 2$ and $|\beta|\leq Q/(qg)$, provided that $Q\leq c^{1/3}$.
Recall that $(c,d)=1$, thus at least one of the above
inequalities is admissible. Therefore, for all $2\leq q\leq Q$, $(a,q)=1$
and $|\beta|\leq Q/(qg)$ we have
\begin{align*}
h\left(-\frac{a}{q}-\beta\right)\ll qd.
\end{align*}
This completes the proof of Lemma~\ref{major arcs2}.
\end{proof}

\begin{lemma}\label{major arcs3}
For any real number $\beta$, we have
\begin{equation*}
\mathcal{N}
=\int_{-\frac{Q}{g}}^{\frac{Q}{g}}F(\beta)h\left(-\beta\right)\dd\beta
+O\big(g^{1/k}(\log g)^3Q^{-2^{1-k}+\varepsilon}+dQ^3\big).
\end{equation*}
\end{lemma}
\begin{proof}
By the orthogonality relation, it is clear that
\begin{equation*}
\mathcal{N}=\int_{-1/2}^{1/2}F(\alpha)h(-\alpha)\dd\alpha=
\int_{\mathfrak{M}(Q)}F(\alpha)h(-\alpha)\dd\alpha+
\int_{\mathfrak{m}(Q)}F(\alpha)h(-\alpha)\dd\alpha.
\end{equation*}
We firstly consider the major arc
\begin{equation*}
\int_{\mathfrak{M}(Q)}F(\alpha)h(-\alpha)\dd\alpha
=\sum_{1\leq q\leq Q}\sum_{\substack{1\leq a\leq q\\ (a,q)=1}}
\int_{-\frac{Q}{qg}}^{\frac{Q}{qg}}
F\left(\frac{a}{q}+\beta\right)
h\left(-\frac{a}{q}-\beta\right)\dd\beta.
\end{equation*}
The term of the above sum for $a=q=1$ equals
\begin{align*}
\int_{-\frac{Q}{g}}^{\frac{Q}{g}}F\left(1+\beta\right)h\left(-1-\beta\right)\dd\beta
=\int_{-\frac{Q}{g}}^{\frac{Q}{g}}F\left(\beta\right)h\left(-\beta\right)\dd\beta.
\end{align*}
The sums for $2\leq q \leq Q$ contributes error terms.
By Lemma~\ref{major arcs2},
we can conclude that
\begin{align*}
\sum_{2\leq q\leq Q}\sum_{\substack{1\leq a\leq q\\ (a,q)=1}}\int_{-\frac{Q}{qg}}^{\frac{Q}{qg}}
F\left(\frac{a}{q}+\beta\right)h\left(-\frac{a}{q}-\beta\right)\dd\beta
\ll\sum_{2\leq q\leq Q}\sum_{\substack{1\leq a\leq q\\ (a,q)=1}}
\int_{-\frac{Q}{qg}}^{\frac{Q}{qg}} gdq~ \dd\beta\ll dQ^3.
\end{align*}

It is now time to estimate the minor arc.
Let $T=g^{1/k}$.
By the Dirichlet approximation theorem, there exist $s\in \mathbb{Z}$ and $t\in \mathbb{N}$ such that
$$(s,t)=1, \quad 1\leq t\leq 2T^{k-1} \quad \text{and} \quad \left|\alpha-\frac{s}{t}\right|\le \frac{1}{2tT^{k-1}}.$$
Then by Lemma 6.4 in \cite{LZ}, for any fixed $\varepsilon>0$ we conclude
\begin{equation}\label{F bound1}
F(\alpha)\ll g^{(1-2^{1-k})/k+\varepsilon}
+\frac{t^\varepsilon g^{1/k}(\log g)}{(t+g|t\alpha-s|)^{2^{1-k}}}.
\end{equation}
If $t > Q$, then
\begin{equation}\label{F bound2}
F(\alpha)\ll g^{1/k}(\log g)Q^{-2^{1-k}+\varepsilon}.
\end{equation}
Similarly, if $1\leq t \leq Q$ and $|\alpha-s/t|> \frac{Q}{tT^k}$,
then~\eqref{F bound2} hold as well. Therefore, we next assume
\begin{equation}\label{min1}
1\leq  t \leq Q \quad \text{and}\quad
 \left|\alpha-\frac{s}{t}\right| \leq \frac{Q}{tT^k}.
\end{equation}
By the Dirichlet approximation theorem, there exist $a\in \mathbb{Z}$ and $q\in \mathbb{N}$ such that
\begin{equation}\label{min2}
(a,q)=1, \quad 1\leq q\leq T^{k/2} \quad \text{and} \quad \left|\alpha-\frac{a}{q}\right|\le \frac{1}{qT^{k/2}}.
\end{equation}
We deduce from \eqref{min1} and \eqref{min2} that
\begin{align*}
\left|\frac{s}{t}-\frac{a}{q}\right|\leq &\left|\alpha-\frac{s}{t}\right|+\left|\alpha-\frac{a}{q}\right|
\\ \leq &\frac{Q}{tT^k} +\frac{1}{qT^{k/2}}
=\frac{1}{tq}\Big(\frac{qQ}{T^{k}} +\frac{t}{T^{k/2}}\Big)
\\ \leq &\frac{Q}{tqT^{k/2}} < \frac{1}{tq}.
\end{align*}
It follows that
\begin{align}\label{type1linkbraq}
s=a \quad \text{and} \quad t=q
\end{align}
We obtain from~\eqref{F bound1} and~\eqref{type1linkbraq} that
\begin{equation}\label{F bound3}
F(\alpha)\ll g^{(1-2^{1-k})/k+\varepsilon}
+\frac{t^\varepsilon g^{1/k}(\log g)}{(q+g|q\alpha-a|)^{2^{1-k}}}.
\end{equation}
Since $\alpha\in\mathfrak{m}(Q)$, we further obtain
\begin{equation*}
F(\alpha)\ll g^{1/k}(\log g)Q^{-2^{1-k}+\varepsilon}.
\end{equation*}
By Lemma~\ref{lemma3.2}, we have
$$\int_{0}^{1}|h(-\alpha)|\dd\alpha\ll (\log g)^2.$$
The treatment of the minor arcs benefits from the following trivial estimates
\begin{equation*}
\begin{split}
 \bigg|\int_{\mathfrak{m}(Q)}F(\alpha)h(-\alpha)\dd \alpha \bigg|
&\ll \sup_{\alpha\in \mathfrak{m}(Q)}|F(\alpha)|\int_{0}^{1}|h(-\alpha)|\dd \alpha\\
&\ll g^{1/k}(\log g)^3Q^{-2^{1-k}+\varepsilon}.
\end{split}
\end{equation*}
This completes the proof of Lemma~\ref{major arcs3}.
\end{proof}

\begin{lemma}\label{elementary lemma}
For any $\ell \leq g$, let $\mathcal{K}_\ell$ be the number of integers
of $0\leq n \leq \ell$ which can be
represented as $cx + dy$ with $x, y \in \mathbb{Z}_{\geq 0}$. Then
\begin{equation*}
\mathcal{K}_\ell=\sum_{i=1}^{\lfloor \ell/d\rfloor}
\left(\left\lfloor\frac{\ell-id}{c}\right\rfloor+1\right).
\end{equation*}
\end{lemma}
\begin{proof}
For any $i\leq \lfloor \ell/d\rfloor$, the number of $j\in\mathbb{Z}_{\geq 0}$,
so that
$$jc + id  \leq  \ell$$
is exactly
$$\left\lfloor\frac{\ell-id}{c}\right\rfloor+1.$$
The lemma now follows immediately by a sum of $j$ from $0$ to $\lfloor \ell/d\rfloor$.
\end{proof}
Then we deduce that the asymptotic formula of $\mathcal{N}$.
\begin{proof}[Proof of Theorem~\ref{thm2}]
By the orthogonality relation, we get
\begin{equation*}
\mathcal{N}=\int_{-1/2}^{1/2}F(\alpha)h(-\alpha)\dd\alpha=
\int_{\mathfrak{M}(Q)}F(\alpha)h(-\alpha)\dd\alpha+
\int_{\mathfrak{m}(Q)}F(\alpha)h(-\alpha)\dd\alpha.
\end{equation*}
We firstly consider the major arc
\begin{equation*}
\int_{\mathfrak{M}(Q)}F(\alpha)h(-\alpha)\dd\alpha
=\sum_{1\leq q\leq Q}\sum_{\substack{1\leq a\leq q\\ (a,q)=1}}
\int_{-\frac{Q}{qg}}^{\frac{Q}{qg}}
F\left(\frac{a}{q}+\beta\right)
h\left(-\frac{a}{q}-\beta\right)\dd\beta.
\end{equation*}
By Theorem 4.1 of~\cite{Vaughan} we have
\begin{equation*}
F(\alpha)=S(q,a)v(\beta)+O(q^{1/2+\varepsilon}(1+g|\beta|)^{1/2}),
\end{equation*}
where
$$S(q,a)=\frac{1}{q}\sum_{n=1}^{q}e\Big(\frac{an^k}{q}\Big)
\quad\text{and}\quad
v(\beta)=\frac{1}{k}\sum_{n\leq g}n^{1/k-1}e(\beta n).$$
Thus
\begin{equation*}
\int_{\mathfrak{M}(Q)}F(\alpha)h(-\alpha)\dd\alpha
=\sum_{1\leq q\leq Q}\sum_{\substack{1\leq a\leq q\\ (a,q)=1}}
S(q,a)\int_{-\frac{Q}{qg}}^{\frac{Q}{qg}}
v(\beta)h\left(-\frac{a}{q}-\beta\right)\dd\beta+\mathfrak{R}(\beta).
\end{equation*}
where the error term $\mathfrak{R}(\beta)$ can be bounded as
\begin{equation*}
\begin{split}
\mathfrak{R}(\beta)&\ll \int_{-\frac{Q}{qg}}^{\frac{Q}{qg}}
q^{1/2+\varepsilon}(1+g|\beta|)^{1/2}\left|h\left(-\frac{a}{q}-\beta\right)\right|\dd\beta\\
&\ll\int_{-\frac{Q}{qg}}^{\frac{Q}{qg}}q^{1/2+\varepsilon}(1+g|\beta|)^{1/2}g\dd\beta\\
&\ll\frac{ Q^{3/2}}{q^{1-\varepsilon}}.
\end{split}
\end{equation*}
By using Lemma~\ref{major arcs2},
we can conclude that
\begin{align*}
\sum_{2\leq q\leq Q}&\sum_{\substack{1\leq a\leq q\\ (a,q)=1}}
S(q,a)\int_{-\frac{Q}{qg}}^{\frac{Q}{qg}}
v(\beta)h\left(-\frac{a}{q}-\beta\right)\dd\beta\\
&
\ll\sum_{2\leq q\leq Q}\sum_{\substack{1\leq a\leq q\\ (a,q)=1}}
\int_{-\frac{Q}{qg}}^{\frac{Q}{qg}} g^{1/k}dq~ \dd\beta\ll dQ^3g^{1/k-1}.
\end{align*}
Let $a=q=1$, we have
\begin{equation}\label{4.1}
\int_{-1/2}^{1/2}F(\alpha)h(-\alpha)\dd\alpha
=\int_{-\frac{Q}{g}}^{\frac{Q}{g}}v(\beta)h(-\beta)\dd\beta+O\left(dQ^3g^{1/k-1}+\frac{Q^{3/2}}{q^{1-\varepsilon}}\right).
\end{equation}
By writing
\begin{equation}\label{4.2}
\int_{-\frac{Q}{g}}^{\frac{Q}{g}}v(\beta)h(-\beta)\dd\beta=
\int_{-1/2}^{1/2}v(\beta)h(-\beta)\dd\beta+\mathfrak{R}_1(\beta)+\mathfrak{R}_2(\beta),
\end{equation}
where
\begin{equation*}
\mathfrak{R}_1(\beta)=\int_{\frac{Q}{g}}^{1/2}v(\beta)h(-\beta)\dd\beta
\quad\text{and}\quad
\mathfrak{R}_2(\beta)=\int_{-1/2}^{-\frac{Q}{g}}v(\beta)h(-\beta)\dd\beta.
\end{equation*}
Note that
\begin{equation*}
\int_{-1/2}^{1/2}v(\beta)h(-\beta)\dd\beta =
\sum_{\substack{0\leq n \leq g\\0\leq x\leq d,~ 0\leq y\leq c}}\frac{n^{1/k-1}}{k}
\int_{0}^{1}e\Big(\beta\big(n-cx-dy\big)\Big)\dd\beta.
\end{equation*}
By the orthogonality relation, the above integral
 equals $1$ if $n=cx+dy$ and $0$ otherwise.
 Hence, by the result of Sylverter~\cite{SJJ}
 as we mentioned at the beginning of the introduction, we get
\begin{align*}
\sum_{\substack{0\leq n\leq g\\0\leq x\leq d,~ 0\leq y\leq c}}
\int_{0}^{1}e\Big(\beta\big(n-cx-dy\big)\Big)\dd\beta
=\frac{g+1}{2}.
\end{align*}
By Lemma~\ref{elementary lemma} and partial summations, we have
\begin{equation*}
\begin{split}
 \sum_{\substack{0\leq n \leq g\\0\leq x\leq d,~ 0\leq y\leq c}}\frac{n^{1/k-1}}{k}&
\int_{0}^{1}e\Big(\beta\big(n-cx-dy\big)\Big)\dd\beta\\
& =\frac{1}{k}\int_{0}^{g}t^{1/k-1}\dd \mathcal{K}_t\\
& =\frac{g^{1/k-1}}{k}\mathcal{K}_g-\frac{1}{k}\int_{0}^{g}
\mathcal{K}_t\dd t^{1/k-1}\\
& =\frac{g^{1/k-1}}{k}\frac{g+1}{2}-\frac{1}{k}\left(1-\frac{1}{k}\right)
\int_{0}^{g}\sum_{i=1}^{\lfloor t/d\rfloor}
\left(\left\lfloor\frac{t-id}{c}\right\rfloor+1\right) t^{1/k-2}\dd t\\
&=\frac{1}{k+1}\frac{g^{1/k-1}}{cd}
+O_k\left(\frac{g^{1/k}}{c}+\frac{g^{1/k}}{d}\right)\\
&=\frac{1}{k+1}\frac{g^{1/k-1}}{cd}
+O_k\left(\frac{g^{1/k}}{c}+\frac{g^{1/k}}{d}\right)\\
&\sim \frac{g^{1/k}}{k+1},\quad \text{as}~c\rightarrow\infty.
\end{split}
\end{equation*}
Suppose that $|\beta|\leq 1/2$. By Lemma 2.8 of~\cite{Vaughan} we have
\begin{equation*}
  v(\beta)\ll \min\{g^{1/k}, |\beta|^{-1/k}\}.
\end{equation*}
For $\frac{Q}{g}\leq \beta\leq \frac12$, using the Lemma \ref{lemma3.2}, we deduce that
\begin{align}\label{4.3}
\mathfrak{R}_1(\beta)&\ll\int_{\frac{Q}{g}}^{1/2}\beta^{-1/k}|h(-\beta)|\dd\beta \leq \left(\frac{g}{Q}\right)^{1/k}\int_{0}^{1}|h(-\beta)|\dd\beta \leq\left(\frac{g}{Q}\right)^{1/k}(\log g)^2,
\end{align}
The same treatment would also lead to the estimate
$\mathfrak{R}_2(\beta)\leq(g/Q)^{1/k}(\log g)^2$.
And then the combination of  \ref{4.1}, \ref{4.2} and \ref{4.3} along with the selected $Q$ from section \ref{S5} yields
 $$\mathcal{N} \sim \frac{g^{1/k}}{k+1},\quad \text{as}~c\rightarrow\infty.$$
This completes the proof of Theorem \ref{thm2}.
\end{proof}

\begin{proof}[Proof of Proposition \ref{proposition 4.1}]
Recall that
\begin{equation*}
\begin{split}
\int_{\mathfrak{M}(Q)}f(\alpha)h(-\alpha)\dd\alpha=\;&\sum_{1\leq q\leq Q}\sum_{\substack{1\leq a\leq q\\ (a,q)=1}}\int_{ \frac{a}{q}-\frac{Q}{qg}}^{ \frac{a}{q}+\frac{Q}{qg}}f(\alpha)h(-\alpha)\dd\alpha \\
=\;&\sum_{1\leq q\leq Q}\sum_{\substack{1\leq a\leq q\\ (a,q)=1}}
\int_{-\frac{Q}{qg}}^{\frac{Q}{qg}}
f\left(\frac{a}{q}+\beta\right)
h\left(-\frac{a}{q}-\beta\right)\dd\beta.
\end{split}
\end{equation*}
In the above expression, the main contribution of the above sum comes form $a=q=1$
\begin{align*}
\int_{-\frac{Q}{g}}^{\frac{Q}{g}}f\left(1+\beta\right)h\left(-1-\beta\right)\dd\beta
=\int_{-\frac{Q}{g}}^{\frac{Q}{g}}f\left(\beta\right)h\left(-\beta\right)\dd\beta.
\end{align*}
By Lemmas~\ref{major arcs1}--\ref{major arcs3}, we have
\begin{align*}
 \int_{\mathfrak{M}(Q)}f(\alpha)h(-\alpha)\dd\alpha
 &=\int_{-1/2}^{1/2}F(\alpha)h(-\alpha)\dd\alpha\\
  &+O\Big(g^{1/k}Q^2\exp\big(-\kappa_1\sqrt{\log g}\big)+
g^{1/k}(\log g)^3Q^{-2^{1-k}+\varepsilon}+dQ^3\Big).
\end{align*}
This completes the proof.
\end{proof}

\section{Proofs of Theorem \ref{thm1} and Theorem \ref{thm3}}\label{S5}
\begin{lemma}\label{lemma5.1}
For positive integers $c$, we have
\begin{align*}
	\sum_{\substack{1\le y\leqslant c\\(y,c)=1}}\frac{y^{1/k}}{\varphi(c)}&=\frac{k}{k+1}c^{1/k}+O(1).
\end{align*}
\end{lemma}
\begin{proof}
By Euler-Maclaurin formula, we obtain that
\begin{eqnarray*}
  \sum_{0<k\le l}m^{1/k} &=&\int_{0}^{l}x^{1/k}dx+ \int_{0}^{l}(x-[x]-\frac{1}{2})\frac{1}{k}x^{1/k-1}dx-(l-[l]-\frac{1}{2})f(l) \\
   &=& \frac{k}{k+1}l^{1+1/k}+O(l^{1/k}).
\end{eqnarray*}
Therefore using the well-known identity of M\"{o}bius' function, we have
\begin{eqnarray*}
   \sum_{\substack{1\le y\leqslant c\\(y,c)=1}}\frac{y^{1/k}}{\varphi(c)}&=& \frac{1}{\varphi(c)}\sum_{y\le c}y^{1/k}\sum_{\substack{ d|y\\d|c}}\mu(d) \\
   &=& \frac{1}{\varphi(c)}\sum_{d|c}\mu(d)\sum_{\substack{ y\le c\\d|y}}y^{1/k} \\
   &=& \frac{1}{\varphi(c)}\sum_{d|c}\mu(d)\left(\frac{k}{k+1}l^{1+1/k}+O(l^{1/k})\right)d^{1/k} \\
   &=&  \frac{k}{k+1}\frac{c^{1+1/k}}{\varphi(c)}\sum_{d|c}\frac{\mu(d)}{d}+\frac{1}{\varphi(c)}\sum_{d|c}\mu(d)O(c^{1/k})\\
   &=&\frac{k}{1+k}c^{1/k}+O(1).
\end{eqnarray*}
\end{proof}
\begin{proof}[Proof of Theorem \ref{thm3}]
By \eqref{eq2-3}, Proposition~\ref{proposition 3.1} and Proposition \ref{proposition 4.1}, we have
\begin{equation*}
 \begin{split}
 \psi_{c,d} =\,&  \mathcal{N}+O\Big( g^{1/k}Q^2\exp\left(-\kappa_1\sqrt{\log g}\right)
 +g^{1/k}(\log g)^3Q^{-2^{1-k}+\varepsilon}+dQ^3  \\
  & + g^{(1-\rho(k))/k+\varepsilon}(\log g)^4
  +g^{1/k}(\log g)^{\delta+4} Q^{-1/2+\varepsilon}\Big),
 \end{split}
\end{equation*}
where $Q\le c^{1/3}$. Let $m=\max\left\{3\cdot2^{k-1},\lceil 2\delta+8\rceil\right\}+1$ and choose $Q=(\log g)^{m}$. Then
\begin{align}\label{eq5-1}
\psi_{c,d}=\mathcal{N}+O\left(\frac{g^{1/k}}{\log g}\right),
\end{align}
provided that $c\ge (\log g)^{m+1}$.

For the complement of our proof, it remains to consider the case that $c\le (\log g)^{m+1}$.
Using Lemma~\ref{lemma5.1}, we have
\begin{align}\label{eq5-2}
	\psi_{c,d}&=\sum_{\substack{n^k=cx+dy\\n^k\leqslant g\\x,y\in \mathbb{Z}_{\geqslant0}}}\Lambda(n)\nonumber\\
	&=\sum_{\substack{1\le y\leqslant c\\(y,c)=1}}\sum_{\substack{n^k\equiv dy\!\!\pmod{c}\\dy\leqslant n^k\leqslant g}}\Lambda(n)+O\left(1\right)\nonumber\\
	&=\sum_{\substack{1\le y\leqslant c\\(y,c)=1}}\left(\psi(g^{1/k};c,dy)-\psi((dy)^{1/k};c,dy)\right)+O\left(1\right)\nonumber\\
	&=\psi(g^{1/k})-\sum_{\substack{1\le y\leqslant c\\(y,c)=1}}\psi((dy)^{1/k};c,dy)+O\left(1\right).
\end{align}
Now, for $c\le (\log g)^{m+1}\ll (\log d)^{m+1}$, by the Siegel--Walfisz theorem we have
\begin{align*}
	\sum_{ \substack{1\le y\leqslant c\\(y,c)=1}}\psi((dy)^{1/k};c,dy)&=\sum_{\substack{1\le y\leqslant c\\(y,c)=1}}\left(\frac{(dy)^{1/k}}{\varphi(c)}
+O\left((dy)^{1/k}\exp(-\kappa_2\sqrt{\log g})\right)\right)\\	&=\frac{k}{k+1}g^{1/k}+O\left(\frac{g^{1/k}}{c^{1/k}}
+g^{1/k}\exp\left(-\kappa_3\sqrt{\log g}\right)\right),
\end{align*}
where $\kappa_2$ and $\kappa_3$ are two positive integers with $\kappa_3<\kappa_2$. Since
$$\psi(g^{1/k})=g^{1/k}
+O\left(g^{1/k}\exp\left(-\kappa_4\sqrt{\log g}\right)\right)$$
by the prime number theorem, we finally conclude from~\eqref{eq5-2} that
\begin{align}\label{eq5-3}
\psi_{c,d}=\frac{1}{k+1}g^{1/k}
+O\left(\frac{g^{1/k}}{c^{1/k}}+g^{1/k}\exp\left(-\kappa_5\sqrt{\log g}\right)\right)
\end{align}	
for $c\le (\log g)^{m+1}$, where $\kappa_5=\min\{\kappa_3,\kappa_4\}$.

From~\eqref{eq5-1} and~\eqref{eq5-3}, we proved that
$$\psi_{c,d}\sim \mathcal{N}, \quad \text{as~}c\rightarrow\infty.$$
This completes the proof of Theorem \ref{thm3}.
\end{proof}

\begin{proof}[Proof of Theorem \ref{thm1}]
For $t\leqslant g$, let
$$\vartheta_{a,b}(t)=\sum_{\substack{p^k=ax+by\\p^k\leqslant t\\x,y\in \mathbb{Z}_{\geqslant0}}}\log p \quad \text{and} \quad \vartheta_{a,b}=\vartheta_{a,b}(g).$$
Integrating by parts, we obtain that
\begin{align}\label{E1}
	\pi_{a,b}=\sum_{\substack{p^k=ax+by\\p^k\leqslant g\\x,y\in \mathbb{Z}_{\geqslant0}}}1
=\frac{\vartheta_{a,b}}{\log g^{1/k}}+\int_{2}^{g^{1/k}}\frac{\vartheta_{a,b}(t)}{t\log ^2t}\dd t.
\end{align}
 By the Chebyshev estimate, we have
 $$\vartheta_{a,b}(t)\leqslant\sum_{p^{1/k}\leqslant t}\log p\ll t^{1/k},$$
 from which it follows that
\begin{equation}\label{E2}
	\int_{2}^{g^{1/k}}\frac{\vartheta_{a,b}(t)}{t\log ^2t}\dd t\ll\int_{2}^{g^{1/k}}
	\frac{t^{1/k}}{t\log^2 t}\dd t\ll\int_{2}^{g^{1/k}}
	\frac{1}{\log^2 t}\dd t\ll\frac{g^{1/k}}{\log^2 g}.
\end{equation}
Again, using the Chebyshev estimate, we have
\begin{equation}\label{E3}
	\vartheta_{a,b}=\psi_{a,b}+O\left(\sqrt{g^{1/k}}\right).
\end{equation}
Thus, by Theorem \ref{thm2} and~\eqref{E1}--\eqref{E3}, we conclude that
\begin{equation*}
	\pi_{a,b}=\frac{\psi_{a,b}}{\log g^{1/k}}+O\left(\frac{\sqrt{g^{1/k}}}{\log g^{1/k}}+\frac{g^{1/k}}{\log^2 g}\right)\sim \frac{\mathcal{N}}{\log g^{1/k}},
\end{equation*}
as $c\rightarrow\infty$.
This completes the proof of Theorem \ref{thm1}.
\end{proof}

\noindent{\bf Conflict of interest.}
The authors state that there is no conflict of interest.



\end{document}